\documentclass{amsart}
\usepackage{hyperref}
\usepackage{color}
\definecolor{mylinkcolor}{rgb}{0.5,0.0,0.0}
\definecolor{myurlcolor}{rgb}{0.0,0.0,0.75}
\hypersetup{colorlinks=true,urlcolor=myurlcolor,citecolor=myurlcolor,linkcolor=mylinkcolor,linktoc=page,breaklinks=true}
\usepackage{fullpage}
\usepackage{amsmath,mathtools,bbm}
\usepackage{booktabs,array}
\usepackage[charter]{mathdesign}
\usepackage{enumerate}

\newcommand{\F}{\mathbf{F}}

\newcommand{\Z}{\mathbf{Z}}

\newcommand{\Q}{\mathbf{Q}}

\renewcommand{\P}{\mathbf{P}}
\newcommand{\Qbar}{\overline{\Q}}

\newcommand{\kbar}{\overline{k}}

\newcommand{\Gal}{{\rm Gal}}
\newcommand{\Jac}{\operatorname{Jac}}

\newcommand{\Pic}{\operatorname{Pic}}
\newcommand{\Div}{\operatorname{Div}}
\renewcommand{\div}{\operatorname{div}}
\newcommand{\ord}{\operatorname{ord}}

\newcommand{\lc}{\operatorname{lc}}
\newcommand{\Res}{\operatorname{Res}}
\newcommand{\opcount}[1]{\hfill\textbf{[#1]}}

\newtheorem{theorem}{Theorem}[section]

\newtheorem{lemma}[theorem]{Lemma}
\newtheorem{proposition}[theorem]{Proposition}

\theoremstyle{definition}
\newtheorem{definition}[theorem]{Definition}

\newtheorem{remark}[theorem]{Remark}

\title{Fast Jacobian arithmetic for hyperelliptic curves of genus 3}
\author{{\small Andrew V. Sutherland}}
\thanks{The author was supported by NSF grants DMS-1115455 and DMS-1522526, and Simons Foundation grant 550033.}

\begin{document}

\begin{abstract}
We consider the problem of efficient computation in the Jacobian of a hyperelliptic curve of genus 3 defined over a field whose characteristic is not 2.
For curves with a rational Weierstrass point, fast explicit formulas are well known and widely available.
Here we address the general case, in which we do not assume the existence of a rational Weierstrass point, using a balanced divisor approach.
\end{abstract}

\maketitle

\section{Introduction}
Like elliptic curves, Jacobians of hyperelliptic curves over finite fields are an important source of finite abelian groups in which the group operation can be made fully explicit and efficiently computed.
This has given rise to many cryptographic applications, including Diffie-Hellman key exchange and pairing-based cryptography, and has also made it feasible to experimentally investigate various number-theoretic questions related to the $L$-series of abelian varieties over number fields, including analogs of the Birch and Swinnerton-Dyer conjecture, the Koblitz-Zywina conjecture, the Lang-Trotter conjecture, and the Sato-Tate conjecture, each of which was originally formulated for elliptic curves but has a natural generalization to abelian varieties of higher dimension.
They can also be used to study analogs of the Cohen-Lenstra heuristics \cite{CL} and related questions in arithmetic statistics that were originally formulated for quadratic number fields but have a natural analog for quadratic function fields \cite{Achter,FW}.

Thanks to work by many authors, there are several algorithms available for Jacobian arithmetic in genus~2 that have been heavily optimized (primarily with a view toward cryptographic applications).
For hyperelliptic curves of genus $g>2$, fully general algorithms have been developed only in the last decade, and fast explicit formulas are typically available only for curves that have a rational Weierstrass point.
This simplifying assumption makes it easier to encode elements of the Jacobian using unique representatives of their divisor class as described by Mumford~\cite{Mumford} and later exploited by Cantor \cite{Cantor}, who gave the first fully explicit algorithm for computing in the Jacobian of a hyperelliptic~curve with a rational Weierstrass~point.

But most hyperelliptic curves do not have a rational Weierstrass point.
Over finite fields the proportion of such curves is roughly $1/(2g)$, and over a number field the proportion is zero (as an asymptotic limit taken over curves of increasing height).  In particular, many arithmetically interesting examples of hyperelliptic curves do not have any rational Weierstrass points. This includes, for example, all 19 of the modular curves $X_0(N)$ that are hyperelliptic.\footnote{This follows from results of Ogg \cite{Ogg74, Ogg78}, who both determined the $N$ for which $X_0(N)$ is hyperelliptic and gave a criterion for rational Weierstrass points on $X_0(N)$ that allows one to rule out the existence of any such points on the hyperelliptic $X_0(N)$.}

In this article we treat hyperelliptic curves of genus $g=3$ over fields whose characteristic is not 2.
Our formulas are based on the \emph{balanced divisor} approach introduced by David J. Mireles Morales in his thesis \cite{Morales} and presented by Galbraith, Harrison, and Mireles Morales in \cite{GHM}.
The basic idea is to represent divisors of degree zero as the difference of an effective divisor of degree $g$ and an effective divisor $D_\infty$ whose support is ``balanced'' over two points at infinity (see \S\ref{section:balanced} for further details).
This is one of two approaches to generalizing Cantor's algorithm; the other is to work in what is known as the \emph{infrastructure} of a ``real'' hyperelliptic curve \cite{JSS,SSW}.
These two approaches were analyzed in \cite{JRS} (using formulas in \cite{EJS,GHM}), which concluded that in genus 2, and even genus in general, the balanced divisor approach is more efficient \cite[\S 8]{JRS}.
When the genus is odd, however, the divisor $D_\infty$ cannot be perfectly balanced; genus~3 thus presents an interesting test case for the balanced divisor approach.

Another reason to be particular interested in the genus~3 case, and the main motivation for this work, is that group computations in the Jacobian play a small but crucial role in efficiently computing the $L$-series of a genus~3 curve.
Recall that for a curve $C/\Q$ we may define its $L$-series as an Euler~product
\[
L(C,s)\coloneqq\prod_pL_p(p^{-s})^{-1},
\]
where $L_p\in\Z[T]$ is an integer polynomial of degree at most $2g$; for primes $p$ of good reduction (all but finitely many), the degree is exactly $2g$ and $L_p(T)$ is the numerator of the zeta function
\[
Z_{C_p}(T)\coloneqq \exp\left(\sum_{r=1}^\infty \#C_p(\F_{p^r})\frac{T^r}{r}\right)=\frac{L_p(T)}{(1-T)(1-pT)},
\]
where $C_p$ denotes the reduction of $C$ modulo $p$.
Using the average polynomial-time algorithm described in \cite{Harvey,HS14,HS16}, for hyperelliptic curves of genus $g$ one can simultaneously compute $L_p(T)\bmod p$ at all primes $p\le N$ of good reduction in time $\tilde{O}(g^3N\log^3 N)$.
In principle one can use a generalization of the algorithm in~\cite{Harvey} to compute $L_p(T)$ modulo higher powers of $p$ sufficient to determine $L_p\in \Z[T]$ (in genus $3$, computing $L_p(T)\bmod p^2$ suffices for $p>144$), but this requires a more intricate implementation and is much more computationally intensive than computing $L_p(T)\bmod p$.

Alternatively, as described in \cite{Elkies,KS}, for curves of genus $3$ one can use $\tilde{O}(p^{1/4})$ group operations in the Jacobian of $C_p$ and its quadratic twist to uniquely determine $L_p\in \Z[T]$.
Within the practical range of computation (where $N\le 2^{32}$, say), the cost of doing this is negligible compared to the cost of computing $L_p(T)\bmod p$, \textit{provided that the group operations can be performed efficiently}.
This is the goal of the present work.

The algorithms we describe here played a key role in \cite{HMS}, which generalizes the algorithm of \cite{HS16} to treat genus 3 curves that are hyperelliptic over~$\Qbar$, but not necessarily over $\Q$ (they may be degree-2 covers of pointless conics).  The output of this algorithm is $L_p(T)L_p(-T)\bmod p$, and, as explained in \cite[\S 7]{HS16}, one can again use $\tilde O(p^{1/4})$ group operations in the Jacobian to uniquely determine $L_p\in\Z[T]$ given this information.
As can be seen in Table 1 of \cite{HMS}, which shows timings obtained using a preliminary version of the formulas presented in this article, the time spent on group operations is negligible compared to the time spent computing $L$-polynomials modulo~$p$ (less than one~tenth).
This was not true of initial attempts that relied on a generic implementation of the balanced divisor approach included in Magma \cite{Magma}, which has not been optimized for hyperelliptic curves of genus 3.
For the application in \cite{HMS}, the primary use of our addition formulas occurs as part of a baby-steps giant-steps search in which field inversions can easily be combined by taking steps in parallel \cite[\S 4.1]{KS}.
The incremental cost of a field inversion is then just three field multiplications, making affine coordinates preferable to projective coordinates (by a wide margin); we thus present our formulas in affine coordinates, although they can be readily converted to projective coordinates if desired.

The explicit formulas we obtain are nearly as fast as the best known formulas for genus~3 hyperelliptic curves that have a rational Weierstrass point \cite{CFADLNV,FWG,GKP,GMACT,KGMCT,Nyukai,NMCT,WPP}, which have been extensively optimized.\footnote{Indeed, our addition formula uses exactly the same number of field inversions and multiplications as the formula in \cite[Alg.\,14.52]{CFADLNV} for genus 3 curves with a rational Weierstrass point in odd characteristic (but as noted above, this formula has since been improved).}
Our formulas for addition/doubling have a cost of \textbf{I+79M}/\textbf{I+82M}, versus \textbf{I+67M}/\textbf{I+68M} for the fastest known formulas for genus 3 hyperelliptic curves with a rational Weierstrass point \cite{Nyukai,NMCT}, where \textbf{I} denotes a field inversion and \textbf{M} denotes a field multiplication.  This performance gap is comparable to that seen in genus 2, where the fastest addition formulas for curves without a rational Weierstrass point cost \textbf{I+28M}/\textbf{I+32M} \cite{GHM}, versus \textbf{I+24M}/\textbf{I+27M} \cite{Lange} for genus 2 curves with a rational Weierstrass point.

Contemporaneous with our work, Rezai Rad \cite{Rad} has independently obtained formulas for genus~3 hyperelliptic curves without a rational Weierstrass point using a modified infrastructure approach that exploits a map from the infrastructure to the Jacobian whose image consists of balanced divisors, obtaining a cost of \textbf{I+75M}/\textbf{I+86M} (for affine coordinates in odd characteristic).
This raises the interesting question of whether it is possible to integrate the faster addition formula in \cite{Rad} with the faster doubling formula presented here.

\section{Background}

We begin by recalling some basic facts about hyperelliptic curves and their Jacobians.

\subsection{Hyperelliptic curves}

A (smooth, projective, geometrically integral) curve $C$ over a field $k$ is said to be \emph{hyperelliptic} if its genus $g$ is at least 2 and it admits a 2-1 morphism $\phi\colon C\to \P^1$ (the \emph{hyperelliptic map}).
The map $\phi$ determines an automorphism $P\to \overline P$ of $C$, the \emph{hyperelliptic involution}, which fixes the fibers of $\phi$ and acts trivially only at ramification points.
The fixed points of the hyperelliptic involution are precisely the \emph{Weierstrass points} of $C$ (the points $P$ for which there exists a non-constant function on~$C$ with a pole of order less than $g+1$ at $P$ and no other poles).
The Riemann-Hurwitz formula implies that a hyperelliptic curve of genus $g$ has exactly $2g+2$ Weierstrass points.
Some authors require the hyperelliptic map $\phi$ to be defined over $k$ (rationally hyperelliptic), while others only require it to be defined over $\kbar$ (geometrically hyperelliptic); we shall assume the former.  When $k$ is a finite field the distinction is irrelevant because~$\P_k^1$ has no non-trivial twists (these would be genus 0 curves with no rational points, which do not occur over finite fields).

Provided $\mathrm{char}(k)\ne 2$, which we henceforth assume, every hyperelliptic curve $C/k$ has an affine model of the form
\[
y^2=f(x),
\]
with $f\in k[x]$ separable of degree $2g+1$ or $2g+2$.
The hyperelliptic map $\phi$ sends each affine point $(x,y)$ on $C$ to $(x:1)$ on $\P^1$, and the hyperelliptic involution swaps $(x,y)$ and $(x,-y)$.
The projective closure of the model $y^2=f(x)$ has a singularity on the line $z=0$ (points on this line are said to lie \emph{at infinity}); the curve $C$ is obtained by desingularization. Equivalently, $C$ is the smooth projective curve with function field $k(C)\coloneqq k[x,y]/(y^2-f(x))$;
the field $k(C)$ is a quadratic extension of the rational function field $k(x)\simeq k(\P^1)$, and the inclusion map $\phi^*\colon k(\P^1)\hookrightarrow k(C)$ corresponds to the hyperelliptic map $\phi$.

When $\deg f=2g+1$, the model $y^2=f(x)$ has a unique rational point at infinity that is also a Weierstrass point.
Conversely, if $C$ has a rational Weierstrass point, we can obtain a model of the form $y^2=f(x)$ with $\deg f=2g+1$ by moving this point to infinity.  We can then make $f$ monic via the substitutions $x\mapsto \lc(f)x$ and $y\mapsto \lc(f)^gy$, after dividing both sides of $y^2=f(x)$ by $\lc(f)^{2g}$.

If $C$ does not have a rational Weierstrass point then we necessarily have $\deg f=2g+2$, and there are either 0 or 2 rational points at infinity, depending on whether the leading coefficient of $f$ is a square in $k^\times$ or not.
Provided that $C$ has some rational point~$P$, moving this point to infinity ensures that there are two rational points at infinity (the other is $\overline P\ne P$).
This makes the leading coefficient of $f$ a square, and we can then make $f$ monic by replacing $y$ with $\sqrt{\lc(f)}y$ and dividing through by $\lc(f)$.

In summary, if $C$ is a hyperelliptic curve with a rational point then it has a model of the form $y^2=f(x)$ with $f$ monic of degree $2g+1$ or $2g+2$.
The former is possible if and only if $C$ has a rational Weierstrass point and the later can always be achieved provided that $C$ has a rational point that is not a Weierstrass point.
If $k$ is a finite field of cardinality $q$, the Weil bound $\#C(k)\ge q +1-2g\sqrt{q}$ guarantees that $C$ has a rational point whenever $q > 4g^2$, and it is guaranteed to have a rational point that is not a Weierstrass point when $q > 4g^2+2g+2$.
For $g=3$ this means that if $k$ is a finite field of odd characteristic and cardinality at least 47, then $C$ has a model of the form $y^2=f(x)$ with $f$ monic of degree 8; in what follows, we shall assume that the hyperelliptic curves $C$ we work with have such a model.
\begin{remark}
In the literature, hyperelliptic curves with a model $y^2=f(x)$ that has two rational points at infinity are sometimes called ``real'' hyperelliptic curves (those with one rational point at infinity are called ``imaginary'').
We avoid this abuse of terminology as it refers to the model and is not an intrinsic property of the curve.
As noted above, in the setting of interest to us every hyperelliptic curve can be viewed as a ``real'' hyperelliptic curve.
\end{remark}

\subsection{Divisor class groups of hyperelliptic curves}

The \emph{Jacobian} of a curve $C/k$ of genus $g$ is an abelian variety $\Jac(C)$ of dimension $g$ that is canonically determined by $C$; see \cite{Milne} for a formal construction.
Describing $\Jac(C)$ as an algebraic variety is difficult, in general, but we are only interested in its properties as an abelian group.  Provided that $C$ has a $k$-rational point, then by \cite[Thm.~1.1]{Milne}, we may functorially identify the group $\Jac(C)$ with the \emph{divisor class group} $\Pic^0(C)$, the quotient of the group $\Div^0(C)$ of divisors of degree~0 by its subgroup of principal divisors.  We recall that a \emph{divisor} on~$C$ can be defined as a formal sum $D=\sum n_P P$ over points $P\in C(\kbar)$ with only finitely $n_P$ nonzero; the \emph{degree} of $D$ is $\deg(D)\coloneqq\sum n_P$.
A divisor is said to be \emph{principal} if it is of the form $\div(\alpha):=\sum_P\ord_P(\alpha)P$ for some function $\alpha\in k(C)$; such divisors necessarily have degree 0.

We are interested in the $k$-rational points of $\Jac(C)$.  Under our assumption that $C$ has a $k$-rational point, these correspond to divisor classes $[D]$ of $k$-rational divisors $D\in \Div^0(C)$ (this means $D=\sum n_PP$ is fixed by $\Gal(\kbar/k)$, even though the points $P$ in its support need not be).
In order to describe the divisor classes $[D]$ explicitly, we now assume that $C$ is a hyperelliptic curve that has a rational point, and fix a hyperelliptic map $\phi\colon C\to \P^1$.
We say that a point $P$ on $C$ is \emph{affine} if it lies above an affine point $(x:1)$ on $\P^1$ and we call $P$ a point \emph{at infinity} if lies above the point $(1:0)$ on $\P^1$.

Recall that a divisor $D=\sum n_P P$ is \emph{effective} if $n_P\ge 0$ for all $P$; an effective divisor can always be written as $\sum_i P_i$, where the $P_i$ need not be distinct.

\begin{definition}
An effective divisor $D=\sum P_i$ on a hyperelliptic curve $C$ is \emph{semi-reduced} if $P_i\ne \overline P_j$ for any $i\ne j$; a semi-reduced divisor whose degree does not exceed the genus of $C$ is said to be \emph{reduced}.
\end{definition}

\begin{lemma}
Let $C/k$ be a hyperelliptic curve that has a rational point.  Every rational divisor class $[D]$ in $\Pic^0(C)$ can be represented by a divisor whose affine part is semi-reduced.
\end{lemma}
\begin{proof}
By adding a suitable principal divisor to $D$ if necessary, we can assume the affine part $D_0$ of $D$ is effective.
If $D_0$ is not semi-reduced it can be written as $D_1+\overline D_1+D_2$ with $D_2$ rational and semi-reduced; if we now take a principal divisor $E$ on $\P^1$ with affine part $\phi_*D_1$ and subtract $\phi^*E$ from~$D$ we obtain a linearly equivalent rational divisor with affine part $D_2$ (here $\phi\colon C\to \P^1$ is the hyperelliptic map).
\end{proof}

Let us now fix a model $y^2=f(x)$ for our hyperelliptic curve $C$ that has a rational point at infinity.
A semi-reduced affine divisor $D=\sum P_i$ can be compactly described by its \emph{Mumford representation} $\div[u,v]$: let $P_i=(x_i,y_i)$, define $u(x)\coloneqq\prod_i(x-x_i)$, and let $v$ be the unique polynomial of degree less than $\deg u$ for which $f-v^2$ is divisible by $u$.  As explained in \cite[\S 1]{Mumford}, this amounts to requiring that $v(x_i)=y_i$ with multiplicity equal to the multiplicity of $P_i$ in $D$; when the $x_i$ are distinct $v$ can be computed via Lagrange interpolation in the usual way.  If $D$ is a rational divisor, then $u,v\in k[x]$.

Conversely, suppose we are given $u,v\in k[x]$ with $u$ monic, $\deg v<\deg u$, and $f-v^2$ is divisible by~$u$.
Write $u(x)=\prod_i(x-x_i)$, define $P_i\coloneqq (x_i,v(x_i))$; the affine points $P_i$ lie in $C(\kbar)$ because $u|(f-v^2)$ implies $f(x_i)-v(x_i)^2$ is divisible by $u(x_i)=0$, and therefore $v(x_i)^2=f(x_i)$.  We now define
\[
\div[u,v] \coloneqq \sum_i P_i.
\]
The effective divisor $\div[u,v]$ is rational, since $u,v\in k[x]$, and it is semi-reduced: if $P_i=\overline P_j$ then we must have $x_i=x_j$ and $v(x_i)=-v(x_j)=-v(x_i)=0$; if $i\ne j$ then $x_i$ is a double root of $u$ and of $v$, and therefore also a double root of $f$, but this is impossible since $f$ is separable.
There is thus a one-to-one correspondence between semi-reduced affine divisors and Mumford representations $\div[u,v]$, and $\div[u,v]$ is rational if and only if $u,v\in k[x]$.

Let us now fix an effective divisor $D_\infty$ of degree $g$ supported on rational points at infinity; if $C$ has one rational point $P_\infty$ at infinity we may take $D_\infty = gP_\infty$, and if $C$ has two rational points $P_\infty$ and $\overline P_\infty$ at infinity we may take $D_\infty=\lceil g/2\rceil P_\infty+\lfloor g/2\rfloor \overline P_\infty$.

\begin{proposition}\label{prop:unique}
Let $C$ be a hyperelliptic curve of genus $g$ and let $D_\infty$ be an effective divisor of degree $g$ supported on rational points at infinity.  Each rational divisor class in $\Pic^0(C)$ can be uniquely written as $[D_0-D_\infty]$, where $D_0$ is an effective rational divisor of degree $g$ whose affine part is reduced.
\end{proposition}
\begin{proof}
See Proposition 1 in \cite{GHM}, which follows from Propositions 3.1 and 4.1 of \cite{PR} (provided the support of $D_\infty$ is rational, which we have assumed).
\end{proof}

\begin{remark}
When $g$ is even it is not actually necessary for the points at infinity to be rational; the divisor $D_\infty=(g/2)(P_\infty+\overline P_\infty)$ will be rational in any case.
Indeed, as astutely observed in \cite{GHM}, when~$C$ has even genus and no rational Weierstrass points, it is computationally advantageous to work with a model for $C$ that does not have rational points at infinity.  But this will not work when the genus is odd because we do need $D_\infty$ to be rational (Proposition~\ref{prop:unique} is false otherwise).
\end{remark}

\section{Hyperelliptic divisor class arithmetic using balanced divisors}\label{section:balanced}

In this section we summarize the general formulas for Jacobian arithmetic using balanced divisors.  Our presentation is based on \cite{GHM}, but we are able to make some simplifications by being more specific about our choice of $D_\infty$ and unraveling a few definitions (we also introduce some new notation).  We refer the reader to \cite{GHM,Morales} for details and proofs of correctness.  In the next section we specialize these formulas to the case $g=3$ and optimize for this case.

Let us first fix a model $y^2=f(x)$ for a hyperelliptic curve $C/k$ of genus~$g$ with rational points $P_\infty\coloneqq (1:1:0)$ and  $\overline{P}_\infty\coloneqq (1:-1:0)$ at infinity (in weighted projective coordinates), and let us define $D_\infty\coloneqq \lceil g/2\rceil P_\infty + \lfloor g/2\rfloor \overline P_\infty$.
This implies that $f$ is monic of degree $2g+2$; as noted above, this can be assumed without loss of generality if $C$ has any rational points that are not Weierstrass points.
The case where $C$ has a rational Weierstrass point is better handled by existing algorithms in any case, so the only real constraint we must impose is that $C$ have a rational point.\footnote{The assumption that $C$ has a rational point is required by any algorithm that represents rational elements of $\Pic^0(C)$ using rational divisors (even though this is not always explicitly stated in the literature).  As observed in \cite[p.\ 287]{Poonen}, without this assumption a rational divisor class need not contain any rational divisors.}
The assumption that $\textrm{char}(k)\ne 2$ is made purely for the sake of convenience, the algorithms in \cite{GHM,Morales} work in any characteristic.

Proposition~\ref{prop:unique} implies that we can uniquely represent each rational divisor class in $\Pic^0(C)$ by a triple $(u,v,n)$, where $\div[u,v]$ is a rational reduced affine divisor in Mumford notation (so $u,v\in k[x]$ satisfy $\deg v < \deg u$, with $u$ a monic divisor of $f-v^2$) with $\deg u\le g$, and $n$ is an integer with $0\le n\le g-\deg u$).
The triple $(u,v,n)$ corresponds to the divisor
\[
\div[u,v,n]\coloneqq \div[u,v]+nP_\infty+(g-\deg u-n)\overline P_\infty - D_\infty.
\]
Whenever we write $\div[u,v,n]$ we assume that $u,v,n$ are as above.
In this notation
\[
\div [1,0,\lceil g/2\rceil] = \div[1,0]+\lceil g/2\rceil P_\infty+(g-0-\lceil g/2\rceil)\overline P_\infty - D_\infty = 0,
\]
is the unique representative of the trivial divisor class in $\Pic^0(C)$.

At intermediate steps in our computations we shall need to work with divisors whose affine parts are semi-reduced but not reduced.
Given a semi-reduced affine divisor $\div[u,v]$ with $\deg u\le 2g$ and an integer $n$ with $0\le n \le 2g-\deg u$, we define
\[
\div[u,v,n]^*\coloneqq \div[u,v]+nP_\infty+(2g-\deg u-n)\overline P_\infty -2D_\infty,
\]
and whenever we write $\div[u,v,n]^*$ we assume that $u,v,n$ are as above (in particular, $\deg u+n \le 2g$).

We begin by precomputing the unique monic polynomial $V$ for which $\deg(f-V^2)\le g$.
This auxiliary polynomial is determined by the top $g+1$ coefficients of $f$ and will be needed in what follows.
\medskip

\noindent
\textbf{Algorithm} \textsc{Precompute}

\noindent
Given $f(x)=x^{2g+2}+f_{2g+1}x^{2g+1}+\cdots+f_1x+f_0$, compute the monic $V(x)$ for which $\deg(f-V^2)\le g$.
\begin{enumerate}[1.]
\item Set $V_{g+1}\coloneqq 1$.
\item For $i=g,g-1,\ldots, 0$ compute $c\coloneqq f_{g+1+i} - \sum_{j=i+1}^{g+1} V_jV_{g+1+i-j}$ and set $V_i\coloneqq c/2$.
\item Output $V(x)\coloneqq x^{g+1}+V_gx^g+\cdots + V_1x+V_0$.
\end{enumerate}
\medskip

We now give the basic algorithm for composition, which is essentially the same as the first step in Cantor's algorithm \cite{Cantor}.  In all of our algorithms, when we write $a\bmod b$ with $a,b\in k[x]$ and $b$ nonzero, we denote the unique polynomial of degree less than $\deg b$ that is congruent to $a$ modulo $b$ (the zero polynomial if $\deg b=0$), and for any divisors $D_1,D_2\in \Div(C)$ we write $D_1\sim D_2$ to denote linear equivalence (meaning that $D_1-D_2$ is principal).
\medskip

\noindent
\textbf{Algorithm} \textsc{Compose}

\noindent
Given $\div[u_1,v_1,n_1]$ and $\div[u_2,v_2,n_2]$, compute $\div[u_3,v_3,n_3]^*$ such that $$\div[u_1,v_1,n_1]+\div[u_2,v_2,n_2]\ \sim\ \div[u_3,v_3,n_3]^*.$$
\begin{enumerate}[1.]
\item Use the Euclidean algorithm to compute monic $w\coloneqq\gcd(u_1,u_2,v_1+v_2)\in k[x]$ and $c_1,c_2,c_3\in k[x]$ such that $w=c_1u_1+c_2u_2+c_3(v_1+v_2)$.
\item Let $u_3\coloneqq u_1u_2/w^2$ and let $v_3\coloneqq (c_1u_1v_2+c_2u_2v_1+c_3(v_1v_2+f))/w \bmod u_3$.
\item Output $\div[u_3,v_3,n_1+n_2+\deg w]^*$.
\end{enumerate}
\medskip

To reduce the divisor $\div[u_3,v_3,n_3]^*$ output by \textsc{Compose} to the unique representative of its divisor class we proceed in two steps.
The first is to repeatedly apply the algorithm below to obtain a divisor whose affine part is semi-reduced with degree at most $g+1$.
\medskip

\noindent
\textbf{Algorithm} \textsc{Reduce}

\noindent
Given $\div[u_1,v_1,n_1]^*$ with $\deg u_1 > g+1$, compute $\div[u_2,v_2,n_2]^*$ with $\deg u_2 \le \deg u_1-2$ such that $$\div[u_1,v_1,n_1]^*\ \sim\ \div[u_2,v_2,n_2]^*.$$
\begin{enumerate}[1.]
\item Let $u_2$ be $(f-v_1^2)/u_1$ made monic and let $v_2\coloneqq -v_1\bmod u_2$.
\item If $\deg v_1=g+1$ and $\lc(v_1)=1$ then let $\delta\coloneqq \deg u_1-(g+1)$;\\
else if $\deg v_1=g+1$ and $\lc(v_1)=-1$ then let $\delta\coloneqq g+1-\deg u_2$; \\
else let $\delta\coloneqq (\deg u_1-\deg u_2)/2$.
\item Output $\div[u_2,v_2, n_1+\delta]^*$.
\end{enumerate}
\medskip

\textsc{Reduce} decreases the degree of the affine part of its input by at least $2$, so at most $\lfloor(g-1)/2\rfloor$ calls to \textsc{Reduce} suffice to reduce the output of \textsc{Compose} to a linearly equivalent divisor whose affine part has degree at most $g+1$.  Having obtained a divisor $\div[u,v,n]^*$ with $\deg u\le g+1$, we need to compute the unique representative of its divisor class.
Now if $\lceil g/2\rceil\le n \le \lceil 3g/2\rceil - \deg u$, then $\deg u \le g$ and
\[
\div[u,v,n]^* = \div[u,v] + (n-\lceil g/2\rceil)P_\infty + (\lceil 3g/2\rceil -\deg u- n)\overline P_\infty + D_\infty - 2D_\infty,
\]
so we can simply take $\div[u,v,n-\lceil g/2\rceil]$ as our unique representative.
The following algorithm ``adjusts'' $\div[u,v,n]^*$ until $n$ is within the desired range; it can be viewed as composition with a principal divisor supported at infinity followed by reduction.
\medskip

\noindent
\textbf{Algorithm} \textsc{Adjust}

\noindent
Given $\div[u_1,v_1,n_1]^*$ with $\deg u_1 \le g+1$ compute $\div[u_2,v_2,n_2]$ such that
$$\div[u_1,v_1,n_1]^*\ \sim\ \div[u_2,v_2,n_2].$$
\begin{enumerate}[1.]
\item If $n_1 \ge \lceil g/2\rceil$ and $n_1 \le \lceil 3g/2\rceil-\deg u_1$ then output $\div[u_1,v_1,n_1-\lceil g/2\rceil]$ and terminate.
\item If $n_1 < \lceil g/2\rceil$, let $\hat v_1\coloneqq v_1 - V + (V \bmod u_1)$, let $u_2$ be $(f-\hat v_1^2)/u_1$ made monic, let $v_2\coloneqq -\hat v_1\bmod u_2$, and let $n_2\coloneqq n_1+g+1-\deg u_2$.
\item If $n_1 \ge \lceil g/2\rceil$, let $\hat v_1\coloneqq v_1 + V - (V \bmod u_1)$, let $u_2$ be $(f-\hat v_1^2)/u_1$ made monic, let $v_2\coloneqq -\hat v_1\bmod u_2$, and let $n_2\coloneqq n_1+\deg u_1-(g+1)$.
\item Output \textsc{Adjust}($\div[u_2,v_2,n_2]^*$).
\end{enumerate}
\medskip

The polynomial $u_2$ computed in step 2 or 3 of \textsc{Adjust} has degree at most $g$ (this is guaranteed by $\deg (f-V^2)\le g$ and $\deg v_1 < \deg u_1$).  If $\deg u_1\le g$ then \textsc{Adjust} either terminates or outputs a value for $n_2$ that is strictly closer to the desired range than $n_1$, and if $\deg u_1=g+1$ then \textsc{Adjust} outputs a divisor whose affine part has strictly lower degree with $n_2$ no further from the desired range than $n_1$.
Thus it always makes progress, and the total number of non-trivial calls to \textsc{Adjust} (those that do not terminate in step 1) is at most $\lceil g/2\rceil+1$.

We now give the general algorithm for adding rational divisor classes.
\medskip

\noindent
\textbf{Algorithm} \textsc{Addition}

\noindent
Given $\div[u_1,v_1,n_1]$ and $\div[u_2,v_2,n_2]$, compute $\div[u_3,v_3,n_3]\ \sim\ \div[u_1,v_1,n_1] + \deg[u_2,v_2,n_2]$.
\begin{enumerate}[1.]
\item Set $\div[u,v,n]^*\leftarrow \textsc{Compose}(\div[u_1,v_1,n_1],\div[u_2,v_2,n_2])$.
\item While $\deg u > g+1$, set $\div[u,v,n]^* \leftarrow \textsc{Reduce}(\div[u,v,n]^*)$.
\item Output $\textsc{Adjust}(\div[u,v,n]^*)$.
\end{enumerate}
\medskip

Note that \textsc{Addition} is fully general; the supports of its inputs may overlap, and it can be used with hyperelliptic curves of any genus, so long as the curve has a model with two rational points at infinity (always true over a sufficiently large finite field).

Let us now analyze the behavior of \textsc{Addition} in the typical case (which will be overwhelmingly dominant when $k$ is a large finite field).  We generically expect divisors to have affine parts of degree~$g$, and even when the two inputs to \textsc{Addition} coincide, we expect the GCD computed in step 1 of \textsc{Compose} to be trivial.

More specifically, we expect the following to occur in a typical call to \textsc{Addition}:
\begin{itemize}
\item The inputs will satisfy $\deg u_1=\deg u_2=g$, $\deg v_1=\deg v_2=g-1$, and $n_1=n_2=0$.
\item The divisor $\div[u,v,n]^*$ output by \textsc{Compose} will have $\deg u=2g$, $\deg v=2g-1$, and $n=0$.
\item Each call to \textsc{Reduce} will reduce the affine degree by 2 and increase $n$ by 1.
\item The input to \textsc{Adjust} will have $\deg u = g+1$ if $g$ is odd, $\deg u =g$ if $g$ is even, and $n=\lfloor g/2\rfloor$.
\item If $g$ is even \textsc{Adjust} will simply set $n$ to 0 and return.  If $g$ is odd \textsc{Adjust} will reduce the degree of $u$ from $g+1$ and increase $n$ by $1$ in the initial call, and then set $n$ to 0 and return.
\end{itemize}

It is worth comparing this to Cantor's algorithm for hyperelliptic curves with a rational Weierstrass point, which instead uses a model $y^2=f(x)$ for $C$ with $\deg f=2g+1$.  If we remove the steps related to maintaining the integers $n$, all of which have negligible cost, the algorithms \textsc{Compose} and \textsc{Reduce} are identical to those used in Cantor's algorithm; the only difference is that in Cantor's algorithm there is no analog of \textsc{Adjust}.
But note that in the typical odd genus case, Cantor's algorithm will need to call \textsc{Reduce} when $\deg u$ reaches $g+1$,
and this is essentially equivalent to calling \textsc{Adjust} in the typical odd genus case.

In summary, the asymptotic complexity of \textsc{Addition} in the typical case is effectively identical to that of Cantor's algorithm; the only meaningful difference is that the degree of the curve equation is $2g+2$ rather than $2g+1$, and this increases the complexity of various operations by a factor of $1+O(1/g)$.

We conclude this section with an algorithm to compute the additive inverse of a divisor class.\footnote{We correct a typo that appears in step 4 of the Divisor Inversion algorithms given in \cite{GHM} and \cite{Morales} ($m_1$ should be $n_1$).}
\medskip

\noindent
\textbf{Algorithm} \textsc{Negation}

\noindent
Given $\div[u_1,v_1,n_1]$, compute $\div[u_2,v_2,n_2]\ \sim\ -\div[u_1,v_1,n_1]$.
\begin{enumerate}[1.]
\item If $g$ is even, output $\div[u_1, -v_1, g-\deg u_1-n_1]$ and terminate.
\item If $n_1 > 0$, output $\div[u_1, -v_1, g-\deg u_1-n_1+1]$ and terminate.
\item Output \textsc{Adjust}($\div[u_1,-v_1, \lceil3 g/2\rceil -\deg u_1 +1]^*$).
\end{enumerate}
\medskip

Perhaps surprisingly, negation is the one operation that is substantially more expensive when the genus is odd (it is trivial when the genus is even).  In the typical case we will have $n_1=0$ and the call to \textsc{Adjust} will need to perform a reduction step.

\section{Explicit formulas in genus 3}

We now specialize to the case $g=3$ and give explicit straight-line formulas for the two most common cases of \textsc{Addition}: adding divisors with affine parts of degree 3 and disjoint support, and doubling a divisor with affine part of degree 3.
We also give a formula for \textsc{Negation} in the typical case.

We assume the curve equation is $y^2=f(x)$ where $f(x)=\sum_{i=0}^8 f_ix_i$ is monic of degree 8 (so $f_8=1$); we also assume that $f_7=0$, which can be achieved via the linear substitution $x\to x-f_7/8$. This implies that our precomputed monic polynomial $V=\sum_{i=0}^4 V_ix^i$ with $\deg (f-V^2)\le 3$ has $V_3=0$.

\subsection{Addition in the typical case}\label{sec:TypicalAddition}

Unraveling the execution of \textsc{Addition} in the typical case for $g=3$ with $\deg u_1=\deg u_2=3$, and $\gcd(u_1,u_2)=1$ yields the following algorithm.
\medskip

\noindent
\textbf{Algorithm} \textsc{TypicalAddition} (preliminary version)

\noindent
Given $\div[u_1,v_1,0]$ and $\div[u_2,v_2,0]$ with $\deg u_1=\deg u_2=3$ and $\gcd(u_1,u_2)=1$, compute $$\div[u_5,v_5,n_5]\sim \div[u_1,v_1,0]+\div[u_2,v_2,0].$$

\begin{enumerate}[1.]
\item Compute $c_1,c_2\in k[x]$ such that $c_1u_1+c_2u_2=1$.
\item Compute $u_3\coloneqq u_1u_2$ and $v_3\coloneqq (c_1u_1v_2+c_2u_2v_1)\bmod u_3$ (we have $\deg u_3=6$ and $n_3=0$).
\item Let $u_4$ be $(f-v_3^2)/u_3$ made monic, and let $v_4\coloneqq -v_3\bmod u_4$ (we have $\deg u_4=4$ and $n_4=1$).
\item Let $\hat v_4\coloneqq v_4-V + (V\bmod u_4)$, let $u_5$ be $(f-\hat v_4^2)/u_4$ made monic, and let $v_5\coloneqq -\hat v_4\bmod u_5$.
\item Output $\div[u_5,v_5,3-\deg u_5]$.
\end{enumerate}
\medskip

As first proposed by Harley in \cite{GH,Harley} for genus 2 curves and subsequently exploited and generalized by many authors, the straight-line program obtained by unrolling the loop in Cantor's algorithm \cite{Cantor} in the typical case can be optimized in two ways. The first is to avoid the GCD computation in step~1 by applying the Chinese remainder theorem to the ring $k[x]/(u_3) = k[x]/(u_1u_2)\simeq k[x]/(u_1)\times k[x]/(u_2)$ to compute
\[
v_3 = \left((v_2-v_1)\,u_1^{-1}\bmod u_2\right)u_1+v_1.
\]
where $u_1^{-1}$ denotes the inverse of $u_1$ modulo $u_2$ (here we use $\gcd(u_1,u_2)=1$).
This expression for $v_3$ has degree at most 5, which is less than $\deg u_3=6$, so there is no need to reduce modulo~$u_1u_2$.

The second optimization is to combine composition with the reduction step, in which we compute $u_4$ as $(f-v_3^2)/u_3$ made monic and $v_4\coloneqq -v_3\bmod u_4$.  If we put $\tilde{s}\coloneqq (v_2-v_1)u_1^{-1}\bmod u_2$, then $u_4$ is
\[
\frac{f-(\tilde su_1+v_1)^2}{u_1u_2} = \frac{(f-v_1^2)/u_1 - \tilde s(\tilde s u_1+2v_1)}{u_2}
\]
made monic.  All the divisions are exact and $u_4$ has degree at most 4, so we only need know the top 3 coefficients of $w\coloneqq (f-v_1^2)/u_1=x^5-u_{12}x^4+(f_6+u_{12}^2-u_{11})x^3+\cdots$, which do not depend on $v_1$ (here we have used $f_7=0$).  To simplify matters we assume $\deg s=2$ (which will typically be true), so that $\deg u_4=4$.
If we let $s$ be $\tilde s$ made monic and put $c\coloneqq 1/\lc(\tilde s)$ and $z\coloneqq su_1$, then
\[
u_4 = (s(z+2cv_2)-c^2w)/u_2\qquad\text{and}\qquad v_4 = -v_1-c^{-1}(z\bmod u_4).
\]

These optimizations are exactly the same as those used to obtain existing explicit formulas that optimize Cantor's algorithm for hyperelliptic curves of genus 3 with a rational Weierstrass point using Harley's approach; see \cite[Alg. 3]{WPP}, for example.
We now discuss a further optimization that is specific to the balanced divisor approach.
Rather than computing $v_4$, we may proceed directly to the computation of $\hat v_4 \coloneqq v_4-V+(V \bmod u_4)$, which is needed to compute $u_5$ as $(f-\hat v_4^2)/u_4$ made monic.
Now $V$ and~$u_4$ are monic of degree 4, so $-V+(V \bmod u_4)=-u_4$ does not depend on $V$, and
\[
\tilde v_4\coloneqq -\hat v_4=u_4-v_4 = u_4+v_1+c^{-1}(z \bmod u_4)
\]
is a monic polynomial of degree 4 that we may use to compute $u_5$ as $(f-\tilde{v}_4^2)/u_4$ made monic and $v_5= \tilde{v}_4\bmod u_5$.

There is a notable difference here with the formulas used for genus 3 hyperelliptic curves with a rational Weierstrass point, where the corresponding expression $(f-v_4^2)/u_4$ is already monic, since $\deg v_4 \le 3$.
But $(f-\tilde{v}_4^2)/u_4$ is not monic; its leading coefficient is $-2\tilde{v}_{43}$, where $\tilde{v}_{43}$ denotes the cubic coefficient of~$\tilde{v}_4$.  Expanding the equations for $u_4,v_4,\tilde{v}_4$ above yields the identity
\begin{equation}\label{eq:vt43add}
\tilde{v}_{43} = u_{12}-u_{22}+c+2s_1+c^{-1}(u_{21}+s_1(s_1-u_{22})-s_0).
\end{equation}

We now give an optimized version of \textsc{TypicalAddition} that forms the basis of our explicit formula.
\medskip

\noindent
\textbf{Algorithm} \textsc{TypicalAddition}

\noindent
Given $\div[u_1,v_1,0]$ and $\div[u_2,v_2,0]$ with $\deg u_1=\deg u_2=3$ and $\gcd(u_1,u_2)=1$, compute $$\div[u_5,v_5,n_5]\ \sim\ \div[u_1,v_1,0]+\div[u_2,v_2,0].$$

\begin{enumerate}[1.]
\item Compute $w \coloneqq (f-v_1^2)/u_1$, and $\tilde s \coloneqq (v_2-v_1)\,u_1^{-1} \bmod u_2$.
\item Compute $c\coloneqq\lc(\tilde s)^{-1}$ and $s=c\tilde s$ and $z \coloneqq su_1$ (require $\deg s = 2$).
\item Compute $u_4\coloneqq(s(z+2cv_1)-c^2w)/u_2$ and $\tilde v_4\coloneqq v_1+u_4+c^{-1}(z\bmod u_4)$.
\item Compute $u_5\coloneqq (2\tilde v_{43})^{-1}(\tilde{v}_4^2-f)/u_4$ and $v_5\coloneqq\tilde v_4\bmod u_5$ (require $\tilde v_{43}\ne 0$).
\item Output $\div[u_5,v_5,3-\deg u_5]$.
\end{enumerate}
\medskip

When expanding \textsc{TypicalAddition} into an explicit formula there are several standard optimizations that one may apply.  These include the use of Karatsuba and Toom style polynomial multiplication, fast algorithms for exact division, the use of Bezout's matrix for computing resultants, and Montgomery's method for combining field inversions.
The last is particular relevant to us, as we require three inversions: the inverse of the resultant $r\coloneqq \Res(u_1,u_2)$ used to compute $u_1^{-1}\bmod u_2$, as well as the inverses of  $\lc(\tilde s)$ and $\tilde{v}_{43}$.
We may use equation \eqref{eq:vt43add} to calculate $\tilde{v}_{43}$ earlier than it is actually needed so that we can invert all three quantities simultaneously using Montgomery's trick: compute $(r\lc(\tilde s)\tilde{v}_{43})^{-1}$ using one field inversion, and then use multiplications to obtain the desired inverses.
We omit the details of these well-known techniques and refer the interested reader to \cite[IV]{WPP}.

An explicit formula that implements \textsc{TypicalAddition} appears in the \hyperref[sec:appendix]{Supplementary Materials} section.  It includes a single exit point where we may revert to the general \textsc{Addition} algorithm if any of our requirements for typical divisors are not met: it verifies the assumptions $\gcd(u_1,u_2)=1$, $\deg s=2$, and $\tilde v_{43}\ne 0$.  This makes it unnecessary to verify $\gcd(u_1,u_2)=1$ before applying the formula.

We give field operation counts for each step in the form $[i\mathbf{I}+m\mathbf{M}+a\mathbf{A}]$, where $i$ denotes the number of field inversions, $m$ is the number of field multiplications (including squarings), and $a$ is the number of additions or subtractions of field elements.
The count $a$ includes multiplications by 2, and also divisions by 2, which can be efficiently implemented using a bit-shift (possibly preceded by an integer addition) and costs no more than a typical field addition.
The divisions by 2 arise primarily in places where we have used Toom-style multiplications and could easily be removed if one wished to adapt the formula to characteristic 2 by switching to Karatsuba.

The total cost of the formula for \textsc{TypicalAddition} is \textbf{I+79M+127A}; this is within 10 or 20 percent of the \textbf{I+67M+108A} cost of the best known formula for addition on genus~3 hyperelliptic curves with a rational Weierstrass point \cite{NMCT} (the exact ratio depends on the cost of field inversions relative to multiplications).\footnote{The formula in \cite{NMCT} contains some typographical errors; see \cite[p.\,25]{FWG} for a clean version.}
Aside from increasing the degree of $f$, the main difference in the two formulas is the need to compute and invert $\hat{v}_{43}$, and to then multiply by this inverse to make $u_5$ monic.
By comparison, the cost of a na\"ive implementation of the unoptimized version of \textsc{TypicalAddition} that uses standard algorithms for multiplication, division with remainder, and GCD (as in \cite[Ch.\,1]{GG}, for example), in which we do not count multiplications or divisions by 1, is \textbf{5I+275M+246A} (c.f. \cite[p.\ 445]{Nagao}).
Our optimizations thus improve performance by a factor of 4 or 5, in terms of the cost of field operations.
In practice the speedup is better than this, closer to $6\times$ when working over word-sized finite fields.
This is due largely to the removal of almost all conditional logic from the explicit formula.

\subsection{Doubling in the typical case}

When doubling a divisor the inputs to \textsc{Addition} are identical, but the GCD computed in \textsc{Compose} is still trivial in the typical case where $\gcd(u_1,v_1)=1$ with $\deg u_1=3$.
The divisor $\div[u_3,v_3,n_3]$ output by \textsc{Compose} will have $u_3=u_1^2$ and $v_3=(c_1u_1v1+c_3(v_1^2+f))\bmod u_1^2$, where $c_1u_1+2c_3v_1=1$.
In this situation we have $v_3\equiv v_1\bmod u_1$, and since both $\div[u_1,v_1]$ and $\div[u_3,v_3]$ are Mumford representations of semi-reduced divisors, we have $u_1|(v_1^2-f)$ and $u_1^2|(v_3^2-f)$.
We may thus view $v_1$ as a square root of $f$ modulo $u_1$, and we may view $v_3$ as a ``lift'' of this square root from $k[x]/(u_1)$ to $k[x]/(u_1^2)$.
Rather than computing $v_3$ as in \textsc{Compose}, as suggested in \cite{GH} we may instead compute it using a single $u_1$-adic Newton iteration:
\[
v_3 :=v_1-\frac{v_1^2-f}{2v_1}\bmod u_1^2.
\]
If we put $w\coloneqq (f-v_1^2)/u_1$ and define $\tilde s\coloneqq w(2v_1)^{-1}\bmod u_1$, where $(2v_1)^{-1}$ denotes the inverse of $2v_1$ modulo $u_1$ (here we use $\gcd(u,v_1)=1$), then $v_3 = v_1+\tilde su_1$, and $u_4$ is
\[
\frac{f-(v_1+\tilde su_1)^2}{u_1^2} = \frac{w-2v_1\tilde s}{u_1} - \tilde{s}^2
\]
made monic.  We now proceed as in \S\ref{sec:TypicalAddition}.
We assume $\deg \tilde s=2$, let $s$ be $\tilde{s}$ made monic, and define $c\coloneqq \lc(\tilde s)^{-1}$ and $z\coloneqq su_1$.
We then have
\[
u_4 = s^2-(c^2w-2cv_1s)/u_1\qquad\text{and}\qquad v_4=-v_1-c^{-1}(z\bmod u_4),
\]
and
\[
\tilde v_4\coloneqq -\hat v_4 = u_4-v_4 = u_4+v_1+c^{-1}(z\bmod u_4)
\]
is a monic polynomial of degree 4 that we may use to compute $u_5$ as $(f-\tilde v_4^2)/u_4$ made monic and $v_5=\tilde v_4\bmod u_5$.
The polynomial $(f-\tilde v_4^2)/u_4$ has leading coefficient $-2\tilde v_{43}$, and expanding the equations for $u_4,v_4,\tilde v_4$ yields the identity
\begin{equation}\label{eq:vt43dbl}
\tilde v_{43} = 2s_1+c+c^{-1}(s_1(s_1-u_{12})-s_0+u_{11}).
\end{equation}
This leads to the following optimized formula for doubling a typical divisor.
\medskip

\noindent
\textbf{Algorithm} \textsc{TypicalDoubling}

\noindent
Given $\div[u_1,v_1,0]$ with $\deg u_1=3$ and $\gcd(u_1,v_1)=1$, compute $$\div[u_5,v_5,n_5]\ \sim\ 2\div[u_1,v_1,0].$$

\begin{enumerate}[1.]
\item Compute $\overline w \coloneqq (f-v_1^2)/u_1\bmod u_1$, and $\tilde s \coloneqq \overline w(2v_1)^{-1} \bmod u_1$.
\item Compute $c\coloneqq \lc(\tilde s)^{-1}$, and $s\coloneqq c\tilde{s}$ and $z\coloneqq su_1$ (require $\deg s=2$).
\item Compute $u_4\coloneqq (c^2w-2csv_1)/u_1-s^2$ and $\tilde v_4 \coloneqq v_1+u_4+c^{-1}(z\bmod u_4)$.
\item Compute $u_5\coloneqq (2\tilde v_{43})^{-1}(\tilde v_4^2-f)/u_4$ and $v_5\coloneqq \tilde v_4\bmod u_5$ (require $\tilde v_{43} \ne 0$).
\item Output $\div[u_5,v_5,3-\deg u_5]$.
\end{enumerate}
\medskip

An explicit formula that implements \textsc{TypicalDoubling} appears in the  \hyperref[sec:appendix]{Supplementary Materials} section.
In terms of field operations, its total cost is \textbf{I+82M+127A}, which may be compared with \textbf{I+68M+102A} for the best known formula for a genus 3 curve with a rational Weierstrass point \cite{NMCT}, and \textbf{5I+285M+258A} for the unoptimized cost of doubling a typical divisor.

\subsection{Negation in the typical case}
Finally, we consider the case of negating a typical divisor $\div[u_1,v_1,0]$ with $\deg u_1=3$, which amounts to computing \textsc{Adjust}($\div[u_1,-v_1,3]^*$).
Let
\[
\tilde v_1  \coloneqq v_1-V+(V\bmod u_1) = -x_4+\tilde v_{12}x^2+\tilde v_{11}x+\tilde v_{10}
\]
(here we have used $V_3=0$).  We wish to compute $u_2$ as $(f-\tilde v_1^2)/u_1$ made monic and $v_2\coloneqq \tilde v_1\bmod u_2$.
The polynomial $(f-\tilde{v}_1)^2/u_1$ has degree 3 and leading coefficient $f_6+2\tilde v_{12}$, where
\[
\tilde v_{12} = v_{12} + u_{12}^2-u_{11}.
\]
We thus obtain the following algorithm.
\medskip

\noindent
\textbf{Algorithm} \textsc{TypicalNegation}

\noindent
Given $\div[u_1,v_1,0]$ with $\deg u_1=3$, compute $\div[u_2,v_2,n_2]\ \sim\ -\div[u_1,v_1,0].$

\begin{enumerate}[1.]
\item Compute $\tilde v_1\coloneqq v_1-V+(V\bmod u_1)$.
\item Compute $u_2 \coloneqq (f_6+2\tilde v_{12})^{-1}(f-\tilde v_1^2)/u_1$ and $v_2 \coloneqq \tilde v_1\bmod u_2$ (require $f_6+2\tilde v_{12}\ne 0$).
\item Output $\div[u_2,v_2,0]$.
\end{enumerate}

\bigskip
\bigskip
\bigskip
\bigskip
\section*{Supplementary material}\label{sec:appendix}

The explicit formulas presented on the following pages were typeset using latex source generated by an automated script that reads an executable version of verified source code; they should thus be free of the typos that unfortunately plague many of the formulas one finds in the literature. Magma source code for the formulas and an implementation of all the algorithms in this article and code to test their correctness can be found at
\begin{center}
\url{https://math.mit.edu/~drew/BalancedDivisor.m}
\end{center}

\begin{table}
\small
\setlength{\extrarowheight}{0.5pt}

\begin{tabular}{|l|}
\hline
\textbf{\textsc{TypicalAddition}}: $\div[u_5,v_5,n_5]\ \sim\ \div[u_1,v_1,0]+\div[u_2,v_2,0]$ with $\gcd(u_1,u_2)=1$.\\
1. Compute $r := \Res(u_1,u_2)$ and $i(x) = i_2x^2 + i_1x + i_0 := ru_1^{-1} \bmod u_2$ (and $w_0:=u_{11}-u_{12}$).\opcount{15M+12A}\\
\hline
$t_{1}:=u_{10}-u_{20};\quad t_{2}:=u_{11}-u_{21};\quad w_{0}:=u_{12}-u_{22};\quad t_{3}:=t_{2}-u_{22}w_{0};$\\
$t_{4}:=t_{1}-u_{21}w_{0};\quad t_{5}:=u_{22}t_{3}-t_{4};\quad t_{6}:=u_{20}w_{0}+u_{21}t_{3};$\\
$i_{0}:=t_{4}t_{5}-t_{3}t_{6};\quad i_{1}:=w_{0}t_{6}-t_{2}t_{5};\quad i_{2}:=w_{0}t_{4}-t_{2}t_{3};$\\
$r:=t_{1}i_{0}-u_{20}(t_{3}i_{2}+w_{0}i_{1});$\\
\hline
2. Compute $q(x) = q_2x^2+q_1x+q_0 := r(v_2-v_1)\,u_1^{-1} \bmod u_2.$\opcount{10M+30A}\\
\hline
$t_{1}:=v_{20}-v_{10};\quad t_{2}:=v_{11}-v_{21};\quad t_{3}:=v_{12}-v_{22};\quad t_{4}:=t_{2}i_{1};\quad t_{5}:=t_{1}i_{0};\quad t_{6}:=t_{3}i_{2};\quad t_{7}:=u_{22}t_{6};$\\
$t_{8}:=t_{4}+t_{6}+t_{7}-(t_{2}+t_{3})(i_{1}+i_{2});\quad t_{9}:=u_{20}+u_{22};\quad t_{10}:=(t_{9}+u_{21})(t_{8}-t_{6});\quad t_{11}:=(t_{9}-u_{21})(t_{8}+t_{6});$\\
$q_{0}:=t_{5}-u_{20}t_{8};$\\
$q_{1}:=t_{4}-t_{5}+(t_{11}-t_{10})/2-t_{7}+(t_{1}-t_{2})(i_{0}+i_{1});$\\
$q_{2}:=t_{6}-q_{0}-t_{4}+(t_{1}-t_{3})(i_{0}+i_{2})-(t_{10}+t_{11})/2;$\\
\hline
3. Compute $t_1 := rq_2\tilde{v}_{43}$ via \eqref{eq:vt43add}, and $w_1 := c^{-1} = q_2/r,\ \ w_2 := c = r/q_2,\ \ w_3:=c^2,\ \ w_4:=(2\tilde{v}_{43})^{-1}$.\\
\hspace{12pt}Then compute $s(x) = x^2+s_1x + s_0:= c(v_2-v_1)\,u_1^{-1} \bmod u_2$ and $\tilde{v}_{43}$.\opcount{I+18M+6A}\\
\hline
$t_{1}:=(r+q_{1})^2+q_{2}(rw_{0}+q_{2}u_{21}-q_{1}u_{22}-q_{0});\quad t_{2}:=2t_{1};\quad t_{3}:=rq_{2};$\\
If $t_2=0$ or $t_3=0$ then abort (revert to \textsc{Addition}).\\
$t_{4}:=1/(t_{2}t_{3});\quad t_{5}:=t_{2}t_{4};\quad t_{6}:=rt_{5};$\\
$w_{1}:=t_{5}q_{2}^2;\quad w_{2}:=rt_{6};\quad w_{3}:=w_{2}^2;\quad w_{4}:=t_{3}^2t_{4};$\\
$s_{0}:=t_{6}q_{0};\quad s_{1}:=t_{6}q_{1};$\\
$\tilde{v}_{43}:=t_{1}t_{5};$\\
\hline
4. Compute $z(x) = x^5+z_4x^4+z_3x^3+z_2x^2+z_1x+z_0:=su_1$.\opcount{4M+15A}\\
\hline
$t_{6}:=s_{0}+s_{1};\quad t_{1}:=u_{10}+u_{12};\quad t_{2}:=t_{6}(t_{1}+u_{11});\quad t_{3}:=(t_{1}-u_{11})(s_{0}-s_{1});\quad t_{4}:=u_{12}s_{1};$\\
$z_{0}:=u_{10}s_{0};\quad z_{1}:=(t_{2}-t_{3})/2-t_{4};\quad z_{2}:=(t_{2}+t_{3})/2-z_{0}+u_{10};\quad z_{3}:=u_{11}+s_{0}+t_{4};\quad z_{4}:=u_{12}+s_{1};$\\
\hline
5. Compute $u_4(x) = x^4+u_{43}x^3+u_{42}x^2+u_{41}x+u_{40}:= (s(z+2cv_1)-c^2(f-v_1^2)/u_1)/u_2.$\opcount{14M+31A}\\
\hline
$u_{43}:=z_{4}+s_{1}-u_{22};$\\
$t_{0}:=s_{1}z_{4};\quad t_{1}:=u_{22}u_{43};$\\
$u_{42}:=z_{3}+t_{0}+s_{0}-w_{3}-u_{21}-t_{1};$\\
$t_{2}:=u_{21}u_{42};\quad t_{3}:=(u_{21}+u_{22})(u_{42}+u_{43})-t_{1}-t_{2};\quad t_{4}:=2w_{2};$\\
$t_{5}:=t_{4}v_{12};\quad t_{6}:=s_{0}z_{3};\quad t_{7}:=(s_{0}+s_{1})(z_{3}+z_{4})-t_{0}-t_{6};$\\
$u_{41}:=z_{2}+t_{7}+t_{5}+w_{3}u_{12}-u_{20}-t_{3};$\\
$u_{40}:=z_{1}+s_{1}(t_{5}+z_{2})+t_{6}+t_{4}v_{11}-w_{3}(f_{6}+u_{12}^2-u_{11})-u_{20}u_{43}-t_{2}-u_{22}u_{41};$\\
\hline
6. Compute $\tilde{v}_4(x) = x^4+\tilde{v}_{43}x^3+\tilde{v}_{42}x^2+\tilde{v}_{41}x+\tilde{v}_{40}:= -\hat{v}_4 =  v_1 + u_4 + c^{-1}(z \bmod u_4)$.\opcount{6M+10A}\\
\hline
$t_{1}:=u_{43}-z_{4}+w_{2};$\\
$\tilde{v}_{40}:=v_{10}+w_{1}(z_{0}+u_{40}t_{1});$\\
$\tilde{v}_{41}:=v_{11}+w_{1}(z_{1}-u_{40}+u_{41}t_{1});$\\
$\tilde{v}_{42}:=v_{12}+w_{1}(z_{2}-u_{41}+u_{42}t_{1});$\\
\hline
7. Compute $u_5(x) = x^3+u_{52}x^2+u_{51}x+u_{50} := (2\tilde v_{43})^{-1}(\tilde{v}_4^2-f)/u_4$. \opcount{9M+17A}\\
\hline
$u_{52}:=\tilde{v}_{43}/2+w_{4}(2\tilde{v}_{42}-f_{6})-u_{43};$\\
$u_{51}:=w_{4}(2(\tilde{v}_{41}+\tilde{v}_{43}\tilde{v}_{42})-f_{5})-u_{52}u_{43}-u_{42};$\\
$u_{50}:=w_{4}(\tilde{v}_{42}^2+2(\tilde{v}_{40}+\tilde{v}_{43}\tilde{v}_{41})-f_{4})-u_{51}u_{43}-u_{52}u_{42}-u_{41};$\\
\hline
8. Compute $v_5(x) = v_{52}x^2+v_{51}x+v_{50}:= \tilde{v}_4 \bmod u_5$.\opcount{3M+6A}\\
\hline
$t_{1}:=u_{52}-\tilde{v}_{43};$\\
$v_{50}:=\tilde{v}_{40}+t_{1}u_{50};$\\
$v_{51}:=\tilde{v}_{41}-u_{50}+t_{1}u_{51};$\\
$v_{52}:=\tilde{v}_{42}-u_{51}+t_{1}u_{52};$\\
\hline
9. Output $\div[u_5,v_5,3-\deg u_5]$.\opcount{Total: I+79M+127A}\\
\hline
\end{tabular}
\end{table}

\begin{table}
\small
\begin{tabular}{|l|}
\hline
\textbf{\textsc{TypicalDoubling}}: $\div[u_5,v_5,n_4]\ \sim\ 2\div[u_1,v_1,0]$ with $\gcd(u_1,v_1)=1$.\\
1. Compute $r := \Res(u_1,v_1)$ and $i(x)=i_2x^2+i_1x+i_0 := rv_1^{-1} \bmod u_1$\  ($w_0:=v_{11}-u_{12}v_{12}$).\opcount{15M+9A}\\
\hline
$w_{0}:=v_{11}-u_{12}v_{12};\quad t_{2}:=v_{10}-u_{11}v_{12};\quad t_{3}:=u_{12}w_{0}-t_{2};\quad t_{4}:=u_{10}v_{12}+u_{11}w_{0};$\\
$i_{0}:=w_{0}t_{4}-t_{2}t_{3};\quad i_{1}:=v_{11}t_{3}-v_{12}t_{4};\quad i_{2}:=v_{11}w_{0}-v_{12}t_{2};$\\
$r:=v_{10}i_{0}-u_{10}(w_{0}i_{2}+v_{12}i_{1});$\\
\hline
2. Compute $p(x)=p_2x^2+p_1x+p_0 := \overline{w} := (f-v_1^2)/u_1 \bmod u_1$ ($w_1:=u_{12}^2,\ w_2:=w_1+f_6$).\ \ \ \opcount{11M+24A}\\
\hline
$w_{1}:=u_{12}^2;\quad t_{2}:=2u_{10};\quad t_{3}:=3u_{11};\quad w_{2}:=w_{1}+f_{6};\quad t_{5}:=2t_{2}-f_{5};\quad t_{6}:=2u_{12};\quad t_{7}:=t_{3}-w_{2};$\\
$p_{2}:=f_{5}+t_{6}(t_{7}-w_{1})-t_{2};$\\
$p_{1}:=f_{4}+u_{12}t_{5}-v_{12}^2-u_{11}(2f_{6}-t_{3})-w_{1}(t_{7}+t_{3});$\\
$p_{0}:=f_{3}-u_{11}(w_{1}t_{6}-t_{5})-t_{2}w_{2}-u_{12}p_{1}-2v_{11}v_{12};$\\
\hline
3. Compute $q(x)= q_2x^2+q_1x+q_0 := r((f-v_1^2)/u_1)v_1^{-1} \bmod u_1$. \opcount{10M+28A}\\
($w_3:=u_{10}+u_{11}+u_{12},\ \ w_4:=u_{10}-u_{11}+u_{12}$)\\
\hline
$t_{1}:=i_{1}p_{1};\quad t_{2}:=i_{0}p_{0};\quad t_{3}:=i_{2}p_{2};\quad t_{4}:=u_{12}t_{3};\quad t_{5}:=(i_{1}+i_{2})(p_{1}+p_{2})-t_{1}-t_{3}-t_{4};\quad t_{6}:=u_{10}t_{5};$\\
$t_{7}:=u_{10}+u_{12};\quad w_{3}:=t_{7}+u_{11};\quad w_{4}:=t_{7}-u_{11};\quad t_{10}:=w_{3}(t_{3}+t_{5});\quad t_{11}:=w_{4}(t_{5}-t_{3});$\\
$q_{0}:=t_{2}-t_{6};$\\
$q_{1}:=t_{4}+(i_{0}+i_{1})(p_{0}+p_{1})+(t_{11}-t_{10})/2-t_{1}-t_{2};$\\
$q_{2}:=t_{1}+t_{6}+(i_{0}+i_{2})(p_{0}+p_{2})-t_{2}-t_{3}-(t_{10}+t_{11})/2;$\\
\hline
4. Compute $t_3 := 2rq_2\tilde{v}_{43}$ via \eqref{eq:vt43dbl}, and $w_5 := 1/c,\ \ w_6:=c, w_7 := 1/\tilde{v}_{43}$.\opcount{I+17M+7A}\\
\hspace{12pt}Then compute $s(x)=x^2+s_1x+s_0 := q/(2r)$ made monic and $\tilde{v}_{43}$.\\
\hline
$t_{0}:=2r;\quad t_{1}:=t_{0}^2;\quad t_{2}:=q_{2}^2;\quad t_{3}:=t_{1}-q_{0}q_{2}+q_{1}(2t_{0}+q_{1}-q_{2}u_{12})+t_{2}u_{11};$\\
If $q_2=0$ or $t_3=0$ then abort (revert to \textsc{Addition}).\\
$t_{4}:=1/(t_{0}q_{2}t_{3});\quad t_{5}:=t_{3}t_{4};\quad t_{6}:=t_{0}t_{5};$\\
$w_{5}:=t_{2}t_{5};\quad w_{6}:=t_{1}t_{5};\quad w_{7}:=t_{1}t_{2}t_{4};$\\
$s_{0}:=t_{6}q_{0};\quad s_{1}:=t_{6}q_{1};\quad \tilde{v}_{43}:=t_{3}t_{5};$\\
\hline
5. Compute $z(x)=x^5+z_4x^4+z_3x^3+z_2x^2+z_1x+z_0 := su_1$.\opcount{4M+12A}\\
\hline
$t_{1}:=w_{3}(s_{0}+s_{1});\quad t_{2}:=w_{4}(s_{0}-s_{1});\quad t_{3}:=u_{12}s_{1};$\\
$z_{0}:=s_{0}u_{10};\quad z_{1}:=(t_{1}-t_{2})/2-t_{3};\quad z_{2}:=(t_{1}+t_{2})/2-z_{0}+u_{10};\quad z_{3}:=u_{11}+s_{0}+t_{3};\quad z_{4}:=u_{12}+s_{1};$\\
\hline
6. Compute $u_4(x)=x^4+u_{43}x^3+u_{42}x^2+u_{41}x+u_{40} := s^2 - (c^2(f-v_1^2)/u_1-2csv_1)/u_1$.\opcount{9M+14A}\\
\hline
$t_{1}:=v_{12}w_{6};\quad t_{2}:=w_{6}^2;$\\
$u_{43}:=2s_{1};$\\
$u_{42}:=2s_{0}+s_{1}^2-t_{2};$\\
$u_{41}:=2(s_{0}s_{1}+u_{12}t_{2}+t_{1});$\\
$u_{40}:=s_{0}^2+2(w_{0}w_{6}+s_{1}t_{1})-t_{2}(w_{2}+2(w_{1}-u_{11}));$\\
\hline
7. $\tilde{v}_4(x) = \tilde{v}_{43}x^3+\tilde{v}_{42}x^2+\tilde{v}_{41}x+\tilde{v}_{40} := -\hat{v}_4 = v_1 + u_4 + c^{-1}(z \bmod u_4)$.  \opcount{6M+10A}\\
\hline
$t_{1}:=u_{43}-z_{4}+w_{6};$\\
$\tilde{v}_{40}:=v_{10}+w_{5}(z_{0}+u_{40}t_{1});$\\
$\tilde{v}_{41}:=v_{11}+w_{5}(z_{1}-u_{40}+u_{41}t_{1});$\\
$\tilde{v}_{42}:=v_{12}+w_{5}(z_{2}-u_{41}+u_{42}t_{1});$\\
\hline
8. $u_5(x)=x^3+u_{52}x^2+u_{51}x+u_{50} := (2\tilde{v}_{43})^{-1}(\tilde{v}_4^2-f)/u_4$.\opcount{7M+17A}\\
\hline
$u_{52}:=\tilde{v}_{43}/2+w_{7}(\tilde{v}_{42}-f_{6}/2)-u_{43};$\\
$u_{51}:=\tilde{v}_{42}+w_{7}(\tilde{v}_{41}-f_{5}/2)-u_{52}u_{43}-u_{42};$\\
$u_{50}:=\tilde{v}_{41}+w_{7}((\tilde{v}_{42}^2-f_{4})/2+\tilde{v}_{40})-u_{51}u_{43}-u_{52}u_{42}-u_{41};$\\
\hline
9. $v_5(x)=v_{52}x^2+v_{41}x+v_{50} := \tilde{v}_4 \bmod u_5$. \opcount{3M+6A}\\
\hline
$t_{1}:=u_{52}-\tilde{v}_{43};$\\
$v_{50}:=\tilde{v}_{40}+t_{1}u_{50};$\\
$v_{51}:=\tilde{v}_{41}-u_{50}+t_{1}u_{51};$\\
$v_{52}:=\tilde{v}_{42}-u_{51}+t_{1}u_{52};$\\
\hline
10. Output $\div[u_5,v_5,3-\deg u_5]$.\opcount{Total: I+82M+127A}\\
\hline
\end{tabular}
\end{table}

\begin{table}
\small
\begin{tabular}{|l|}
\hline
\textbf{\textsc{TypicalNegation}}: $\div[u_2,v_2,0]\ \sim \ -\div[u_1,v_1,0]$.\\
1. Compute $\tilde{v}_1(x) =-x^4+\tilde{v}_{12}x^2+\tilde{v}_{11}x+\tilde{v}_{10} := v_1  - V + (V \bmod u_1)$.\opcount{3M+5A}\\
\hline
$\tilde{v}_{12}:=v_{12}-u_{11}+u_{12}^2;$\\
$\tilde{v}_{11}:=v_{11}-u_{10}+u_{11}u_{12};$\\
$\tilde{v}_{10}:=v_{10}+u_{10}u_{12};$\\
\hline
2. Compute $u_2(x) = x^3+u_{22}x^2+u_{21}x+u_{20} := (f_6+2\tilde{v}_{12})^{-1}(f-\tilde{v}_1^2)/u_1$.\hspace{16pt}\opcount{I+8M+14A}\\
\hline
$t_{1}:=2\tilde{v}_{12};\quad t_{2}:=f_{6}+t_{1};$\\
If $t_1 = 0$ then abort (revert to \textsc{Negation}).\\
$t_{3}:=1/t_{2};$\\
$u_{22}:=t_{3}(f_{5}+2\tilde{v}_{11})-u_{12};$\\
$u_{21}:=t_{3}(f_{4}+2\tilde{v}_{10}-\tilde{v}_{12}^2)-u_{11}-u_{12}u_{22};$\\
$u_{20}:=t_{3}(f_{3}-t_{1}\tilde{v}_{11})-u_{10}-u_{11}u_{22}-u_{12}u_{21};$\\
\hline
3. Compute $v_2(x) = v_{22}x^2+v_{21}x+v_{20}:= \tilde{v}_1 \bmod u_2$.\opcount{3M+5A}\\
\hline
$v_{22}:=\tilde{v}_{12}-u_{22}^2+u_{21};$\\
$v_{21}:=\tilde{v}_{11}-u_{21}u_{22}+u_{20};$\\
$v_{20}:=\tilde{v}_{10}-u_{20}u_{22};$\\
\hline
4. Output $\div[u_2,v_2,0]$.\opcount{Total: I+14M+24A}\\
\hline
\end{tabular}
\end{table}

\end{document}